\documentclass[a4paper,12pt]{article}
\usepackage{amsmath,amsthm,amssymb}
\usepackage{graphicx}
\usepackage[margin=1in]{geometry}

\newtheorem{conjecture}{Conjecture}

\newtheorem{observation}{Observation}

\newtheorem{proposition}{Proposition}

\newtheorem{theorem}{Theorem}

\DeclareMathOperator{\conv}{conv}

\DeclareMathOperator{\pos}{pos}
\DeclareMathOperator{\inte}{int}

\DeclareMathOperator*{\argmax}{arg\,max}

\begin{document}
\title{Hadwiger's conjecture holds for strongly monotypic polytopes}
\author{Vuong Bui\thanks{LIRMM, Universit\'e de Montpellier (\texttt{bui.vuong@yandex.ru})}}
\date{}

\maketitle
\begin{abstract}
    In this short note, we prove Hadwiger's conjecture for strongly monotypic polytopes.
\end{abstract}

\section{Introduction}
Hadwiger's conjecture is an old and curious problem on the covering of a convex body $C$ by a number of translates of a slightly smaller homothety $(1-\epsilon)C$. The necessary number of translates in the worst case was conjectured to be $2^n$ by Hadwiger \cite{hadwiger1972ungeloste} in 1972, where $n$ is the dimension of $C$. While it was proved to hold for every convex body on the plane by Levi \cite{levi1955uberdeckung} in 1955, little is known even in dimension $3$. (We should remark that Levi also studied the problem before Hadwiger, and it is sometimes called the Levi--Hadwiger conjecture.) In particular, we only know that the conjecture holds for centrally symmetric polyhedra \cite{lassak1984centrally} and bodies of constant width in $\mathbb R^3$ \cite{lassak1997illumination}, and zonotopes for every dimension \cite{martini1985some} (see also the survey in \cite[Chapter 3]{brass2005research} for some more information). One has the right to doubt the conjecture, as the evidences so far are not rich, while Borsuk's conjecture, a problem on the covering of convex bodies with respect to the diameters, was disproved in rather high dimensions \cite{kahn1993counterexample}.

In this article, we try to supply more evidences by proving Hadwiger's conjecture for strongly monotypic polytopes. As one may guess that Hadwiger's conjecture for polytopes may be not much easier than that for all convex bodies, simple polytopes appear to be the candidate to tackle. We however even narrow down the scope further to strongly monotypic polytopes, a subclass of monotypic polytopes. These notions were first introduced in \cite{mcmullen1974monotypic}: A polytope $P$ is monotypic if every polytope with same set of normals as $P$ is combinatorially equivalent to $P$. One can see that a monotypic polytope is necessarily a simple polytope, but the converse is not always true (e.g. we translate a facet of a square pyramid so that the apex is the intersection of only $3$ facets to obtain a simple polytope but not a monotypic polytope). On the other hand, a polytope $P$ is strongly monotypic if any other polytope $Q$ with the same set of normals as $P$ satisfies that the arrangements of the hyperplanes containing the facets of $P$ and $Q$ are combinatorially equivalent. One would notice that the monotypy and strong monotypy of polytopes depend only on the sets of normals (one can see that Theorem \ref{thm:strmono-char} and Theorem \ref{thm:mono-char} emphasize this). A recent interesting fact is that strongly monotypic polytopes are equivalent to polytopes with the generating property \cite{bui2023every}. (We remind that a convex body $C$ has the generating property if for every nonempty intersection $I$ of translates of $C$, there exists a convex body $J$ so that the Minknowskii sum $I+J$ is $C$.) Examples of strongly monotypic polytopes can be found in the original article \cite{mcmullen1974monotypic}, which partially characterized these polytopes in $\mathbb R^3$. All strongly monotypic polytopes in $\mathbb R^3$ were later fully characterized in \cite{borowska2008strongly}.

At first, we should state precisely Hadwiger's conjecture.
\begin{conjecture}[Hadwiger's conjecture \cite{hadwiger1972ungeloste}]
    Every convex body $C$ in $\mathbb R^n$ can be covered by at most $2^n$ translates of $(1-\epsilon) C$ for some $\epsilon>0$.
\end{conjecture}

There is an equivalent problem to Hadwiger's conjecture that was introduced by Boltyanski.
\begin{conjecture}[Boltyanski's illumination conjecture \cite{boltyanski1960problem}]
    For every convex body $C$, there exists a set $V$ of at most $2^n$ vectors so that for every point $x$ on the boundary of $C$, there is a vector $v\in V$ so that $x-v$ is in the interior of $C$.
\end{conjecture}

A discussion on the equivalence of the two conjectures can be found in \cite[Chapter 3]{brass2005research}. We use the illumination conjecture of Boltyanski in the proofs.

\section{The proof for strongly monotypic polytopes}
The following characterization of monotypic polytopes was given in the same article \cite{mcmullen1974monotypic}, which introduces the notion. (The original definition is still by the combinatorial equivalence of polytopes taking the same set of normals.)

\begin{theorem}[McMullen--Schneider--Shephard 1974 \cite{mcmullen1974monotypic}]
    A polytope $P$ is monotypic if and only if every two disjoint primitive subsets $V_1,V_2$ of the set of normals $N(P)$ satisfies $\pos V_1\cap \pos V_2=\{0\}$.
\end{theorem}
Note that a subset of normals $V\subseteq N(P)$ is primitive if the normals in $V$ are linearly independent and the positive hull of $V$ is empty of other normals of $P$.

The combinatorial equivalence between monotypic polytopes of the same set of normals suggests the following observation. It is a corollary of the above characterization.
\begin{proposition} \label{prop:unique-fan}
    There is only one simplicial fan by the normals of a monotypic polytope $P$, which is therefore also the normal fan of $P$.
\end{proposition}
\begin{proof}
    Suppose we have two different simplicial fans, that is there are two distinct primitive subsets $V_1,V_2$ of $N(P)$ so that the positive hulls of $V_1$ and $V_2$ intersect at a point in the relative interior of both, that is
    \[
        \sum_{x_i\in V_1} \lambda_i x_i = \sum_{y_j\in V_2} \theta_j y_j
    \]
    where all the coefficients $\lambda_i,\theta_j$ are positive.
    
    We can assume that $V_1,V_2$ are disjoint, since otherwise, say $x_{i^*}=y_{j^*}$, we can remove it from one of two sides, or both, by
    \[
        (\lambda_{i^*} - \min\{\lambda_{i^*},\theta_{j^*}\}) x_{i^*} + \sum_{x_i\in V_1\setminus\{x_{i^*}\}} \lambda_i x_i = (\theta_{j^*} - \min\{\lambda_{i^*},\theta_{j^*}\}) y_{j^*} + \sum_{y_j\in V_2\setminus\{y_{j^*}\}} \theta_j y_j,
    \]
    where the coefficients $\lambda_{i^*} - \min\{\lambda_{i^*},\theta_{j^*}\}$ and $\theta_{j^*} - \min\{\lambda_{i^*},\theta_{j^*}\}$ are still nonnegative with at least one of them being zero.
    After removing, the new subsets (containing the points with positive coefficients) are still distinct and primitive, but with one less point in the intersection. (Note that the positive hulls still intersect each other.)
    
    However, when $V_1,V_2$ are both primitive and disjoint, we have a contradiction to the monotypy of $P$. Therefore, every simplicial fan by the normals $N(P)$ is identical. The conclusion follows.
\end{proof}

In \cite{bui2023every}, another characterization was given for monotypic and strongly monotypic polytopes, with the one for strongly monotypy given below. (The one for monotypy is given in Theorem \ref{thm:mono-char}.) 
\begin{theorem}[Bui 2023 \cite{bui2023every}]
\label{thm:strmono-char}
    An $n$-dimensional polytope $P$ is strongly monotypic if and only if every $n+1$ normals of $P$ are not in conical position.
\end{theorem}

A clarification on terminology: A set of points is said to be in \emph{conical position} if it is separated from $0$ and none of the points is in the positive hull of the others. (A set of points is said to be separated from $0$ if there is a hyperplane strictly separating the set from $0$.)

One can easily verify the following observation.
\begin{observation}
\label{obs:signs}
    Suppose we have $n$ linearly independent points $a_1,\dots,a_n$ and a point $x$ in the linear span of these points with the unique representation $x=\sum_i\lambda_i a_i$. Then we have:
    \begin{itemize}
        \item The points $x,a_1,\dots,a_n$ are not separated from $0$ if and only if all the coefficients $\lambda_i$ are nonpositive.
        \item One of the points $x,a_1,\dots,a_n$ is in the positive hull of the others if and only if either (i) all the coefficients $\lambda_i$ are nonnegative or (ii) some $\lambda_i>0$ and all other $\lambda_j\le 0$ ($j\ne i$).
    \end{itemize}
\end{observation}
We omit the proof since it is obvious by the uniqueness of the representation

While the characterization in Theorem \ref{thm:strmono-char} is somewhat local, it actually allows us to specify a global picture of the normals of a strongly monotypic polytope.
\begin{theorem} \label{thm:skeleton}
        The set $X$ of the normals of an $n$-dimensional strongly monotypic polytope $P$ contains some $k$ disjoint subsets
	$X_1,\dots, X_k$ such that 
 
        (i) each $X_i$ is the set of
	vertices of a simplex whose relative interior contains $0$,
	
        (ii) the linear spaces spanned by each $X_i$ are linearly
	independent,
        
        (iii) these linear spaces directly sum up to
	$\mathbb R^n$.  
\end{theorem}
\begin{proof}
    Among the elements of $X$ take a set of points $B$ such that $B$ is in
    conical position, $B$ spans an $n$-dimensional space and its positive
    hull is maximal in the sense that no other such points have the
    positive hull being a strict superset of the positive hull of $B$.  Since
    every $n+1$ points of $X$ are not in conical position by Theorem \ref{thm:strmono-char}, it follows that $B$ is a set of $n$ linearly independent points $b_1,\dots,b_n$.
    
    Consider any other point $x=\sum_{i=1}^n \lambda_i b_i$ in $X$.  The
    $n+1$ points $x,b_1,\dots,b_n$ either (iv) have the convex hull
    containing $0$, for which all $\lambda_i$ are nonpositive or (v) have
    one of them in the positive hull of the others. In the case (v), we can conclude that all $\lambda_i$ are nonnegative,  
    since we cannot have the other possibility as in Observation \ref{obs:signs} that
    there is only one positive coefficient, say $\lambda_i > 0$, and
    other nonpositive coefficients. Indeed, suppose otherwise, we have
    $b_i\in\pos(\{x,b_1,\dots,b_n\}\setminus\{b_i\})$, which is a
    contradiction by the strict inclusion
    \[
        \pos(\{x,b_1,\dots,b_n\}\setminus\{b_i\})\supset \pos\{b_1,\dots,b_n\}.
    \]
    Note that the former positive hull contains $x$ while the latter does not (at least some $\lambda_j<0$ with $j\ne i$, otherwise $x$ and $b_i$ are the same normal). So, it is always either (iv) with all nonpositive $\lambda_i$ or (v) with all nonnegative $\lambda_i$.
    
    Let the Cartesian support of a point $x$ be the set $\{b_i: \lambda_i\ne 0\}$ for
    $x=\sum_{i=1}^n \lambda_i b_i$.
    
    We show that the Cartesian supports of two points of $X$ both in
    $\pos\{b_1,\dots,b_n\}$ or both in $\pos\{-b_1,\dots,-b_n\}$ are
    either disjoint or one is a subset of the other.  Let the two points
    be $x,y$ with $x=\sum_{i=1}^n \lambda_i b_i$ and $y=\sum_{i=1}^n
    \theta_i b_i$.  Suppose the Cartesian support of $x$ be
    $\{b_1,\dots,b_{k_2}\}$ and the Cartesian support of $y$ be
    $\{b_{k_1},\dots,b_{k_3}\}$ with $1<k_1\le k_2<k_3$.  Since the points
    in $\{y,b_1,\dots,b_n\}\setminus\{b_{k_1}\}$ are linearly independent,
    the linear relation
    \[
    	x=\left(\sum_{i=1}^{k_1-1} \lambda_i b_i\right) +
    	\frac{\lambda_{k_1}}{\theta_{k_1}} y + \sum_{i=k_1+1}^{k_2}
    	\left(\lambda_i - \frac{\lambda_{k_1}}{\theta_{k_1}} \theta_i\right) b_i
    	- \sum_{i=k_2+1}^{k_3} \frac{\lambda_{k_1}}{\theta_{k_1}}
    	\theta_i b_i
    \]
    is unique.  It raises a contradiction as the $n+1$
    points $\{x,y\}\cup\{b_1,\dots,b_n\}\setminus\{b_{k_1}\}$ are in conical position, with either two positive and one negative coefficients (if $x,y\in\pos\{b_1,\dots,b_n\}$) or two negative and one positive coefficients (if $x,y\in\pos\{-b_1,\dots,-b_n\}$) of the points $b_1, y, b_{k_2+1}$ on
    the right hand side (by Observation \ref{obs:signs}).
    
    Since $0\in\inte\conv X$, there must be some $k$ points
    $x_1,\dots,x_k$ all in $\pos\{-b_1,\dots,-b_n\}$ and their Cartesian
    supports are disjoint while the union of the Cartesian supports is
    $\{b_1,\dots,b_n\}$.  Let $X_i$ for each $i=1,\dots,k$ be the union of
    the Cartesian support of $x_i$ and the point $x_i$ itself, we obtain
    the desired sets $X_1,\dots,X_k$, which complete the proof.
\end{proof}

Now we are ready to prove Hadwiger's conjecture for strongly monotypic polytopes.
\begin{theorem}
\label{thm:main}
    For every strongly monotypic polytope $P$ with the sets $X_1,\dots,X_k$ as in Theorem \ref{thm:skeleton}, there exists a set $V$ of $|X_1|\dots|X_k|$ vectors so that for every point $x$ on the boundary of $P$, there is a vector $v\in V$ so that $x-v$ is in the interior of $P$. In particular, Hadwiger's conjecture holds for strongly monotypic polytopes since $|X_1|\dots|X_k|\le 2^n$ for $P\subset\mathbb R^n$.
\end{theorem}
Note that the proof below uses Boltyanski's illumination version of the conjecture.
\begin{proof}
    For each $x_1\in X_1,\dots,x_k\in X_k$, we consider the cone $\mathcal C = \pos\{(X_1\setminus\{x_1\})\cup\dots\cup (X_k\setminus\{x_k\})\}$. The union of these $q=|X_1|\dots|X_k|$ cones $\mathcal C_1,\dots,\mathcal C_q$ is $\mathbb R^n$ and they are disjoint except at the boundaries. For each cone $\mathcal C_i$, we have a vector $v_i$ so that $\langle v_i, y\rangle > 0$ for any $y\in \mathcal C_i$.
    
    In each cone $\mathcal C_i$, we continue to triangulate further into smaller simplicial cones with the normals in $N(P)\cap \mathcal C_i$. There is a unique way to triangulate these cones $\mathcal C_1,\dots,\mathcal C_q$ due to Proposition \ref{prop:unique-fan}, which also states that the resulting simplicial fan is also the normal fan of $P$. Each cone $K$ in the resulting simplicial fan corresponds to a face $F$ of $P$, in the sense that
    \[
        K=\{\vec{n}\in\mathbb R^n \mid F\subseteq\argmax_{x\in P} \langle \vec{n},x\rangle\}.
    \]

    To finish the proof, we show that the set $V=\{\varepsilon v_1,\dots,\varepsilon v_q\}$ for any sufficiently small $\varepsilon$ satisfies the requirements of the theorem.

    Given some $\delta>0$, we define $H_\delta(x)$ for a point $x$ on the boundary of $P$ to be the set of normals $\vec{n}\in N(P)$ so that the distance of $x$ to the facet of $\vec{n}$ is at most $\delta$, that is
    \[
        H_\delta(x) = \{\vec{n}\in N(P): h_P(\vec{n}) - \langle \vec{n},x\rangle \le \delta\},
    \]
    where we use the standard notation $h_P(\vec{n})=\sup_{p\in P} \langle \vec{n}, p\rangle$.

    We can choose $\delta$ sufficiently small so that $H_\delta(x)$ for any $x\in\partial P$ is contained in a cone of the normal fan. Indeed, assume we cannot choose such a $\delta$. It follows that we have a sequence
$\delta_k\to 0$ and a sequence $x_k\in \partial P$ such that
$H_{\delta_k}(x_k)$ is not contained in a cone of the normal fan.
Passing to a subsequence we assume that $x_k\to x_0\in\partial P$.
For any normal $\vec n\in N(P)$ such that $h_P(\vec n) > \langle \vec n,
x_0\rangle$, the inequality $h_P(\vec n) - \langle \vec n, x_k\rangle \le \delta_k$
tends to the false inequality $h_P(\vec n) - \langle \vec n, x_0\rangle \le 0$
and therefore fails for sufficiently large $k$. Since the set $N(P)$ is
finite, all such inequalities fail for sufficiently large $k$, which
effectively means that $H_{\delta_k}(x_k)$ is contained in the normal
cone of the point $x_0$.
    
    Now, consider any vertex $x$ on the boundary of $P$.
    Let $K$ be the cone of the normal fan that contains $H_\delta(x)$. Let $\mathcal C_j$ contain $K$.
    For every $\vec{n}\in H_\delta(x)\subseteq K\subseteq\mathcal C_j$ and every $\varepsilon>0$, we have
    \[
        \langle \vec{n}, x-\varepsilon v_j\rangle = \langle \vec{n},x\rangle - \varepsilon \langle \vec{n}, v_j\rangle < \langle \vec{n},x\rangle.
    \]
    If $x-\varepsilon v_j$ is not in the interior of $P$, it is due to another normal $\vec{m}\in N(P)\setminus H_\delta(x)$, in the sense that
    \[
        \langle \vec{m}, x-\varepsilon v_j\rangle \ge h_P(\vec{m}).
    \]
    It will not happen if either $\langle \vec{m}, v_j\rangle \ge 0$ or
    \[
        \varepsilon < \frac{h_P(\vec{m}) - \langle \vec{m},x\rangle}{-\langle \vec{m}, v_j\rangle}.
    \]
    Since $h_P(\vec{m}) - \langle \vec{n},x\rangle > \delta$, it suffices to set
    \[
        \varepsilon = \frac{\delta}{\min\limits_{\vec{n}\in N(P),j=1,\dots,q:\ \langle \vec{n}, v_j\rangle < 0} |\langle \vec{n}, v_j\rangle|}.
    \]
    The set $V=\{\varepsilon v_1,\dots,\varepsilon v_q\}$ with this value of $\varepsilon$ satisfies the requirements.
\end{proof}

\section{Some discussions on monotypic polytopes}
Hadwiger's conjecture for all polytopes or more generally for all convex bodies may ask for more techniques than what have been presented in this article. However, it seems that the approach in this article can be readily applied to all monotypic polytopes, which are not necessarily strongly monotypic. 
Note that Proposition \ref{prop:unique-fan} applies to all monotypic polytopes, for which we know that every simplicial fan is identical to the normal fan. What remains is to establish a ``skeleton'' as in Theorem \ref{thm:skeleton}. This may be the hardest part, while the arguments in the proof of Theorem \ref{thm:main} may be mostly kept the same. To close the article, we mention the characterization of monotypic polytopes in \cite{bui2023every}, in corresponding to Theorem \ref{thm:strmono-char}, which may help establishing the ``skeleton''.

\begin{theorem}[Bui 2023 \cite{bui2023every}]
\label{thm:mono-char}
    The monotypy of an $n$-dimensional polytope $P$ is equivalent to: If some $n+1$ normals of $P$ are in conical position, then their positive hull contains another normal of $P$.
\end{theorem}

\subsection*{Acknowledgements}
The author is grateful to Roman Karasev for his valuable comments on the manuscript, in particular for the argument of the existence of $\delta$.

\bibliographystyle{unsrt}
\bibliography{strmonohadwig}

\end{document}